\documentclass{mcrfOF}
\usepackage{amsmath}
\usepackage{amssymb}
  \usepackage{graphics} 
  \usepackage{epsfig} 
\usepackage{graphicx}  \usepackage{epstopdf}
 \usepackage[colorlinks=true]{hyperref}
\hypersetup{urlcolor=blue, citecolor=red}

\def\currentvolume{X}
 \def\currentissue{X}
  \def\currentyear{2019}
   \def\currentmonth{XX}
    \def\ppages{X--XX}
          \def\DOI{10.3934/mcrf.2019033}

\newtheorem{theorem}{Theorem}[section]

\newtheorem{lemma}[theorem]{Lemma}

\theoremstyle{definition}
\newtheorem{definition}[theorem]{Definition}
\newtheorem{remark}{Remark}

\def\dref#1{(\ref{#1})}

\title[Regional gradient controllability of ultra-slow diffusions ...] 
      {Regional gradient controllability of ultra-slow diffusions involving the Hadamard-Caputo time fractional derivative}

\author[Ruiyang Cai, Fudong Ge, YangQuan Chen and Chunhai Kou]{}


 \email{2151359@mail.dhu.edu.cn}
 \email{gefd@cug.edu.cn}
 \email{ychen53@ucmerced.edu}
 \email{kouchunhai@dhu.edu.cn}

\thanks{The second author is supported by the Fundamental Research Funds for the Central Universities, China University of Geosciences, Wuhan (No.CUG170627), the Natural Science Foundation of China (NSFC,  No.KZ18W30084) and the Hubei NSFC (No.2017CFB279).}

\thanks{$^*$ Corresponding author.}

\begin{document}

\maketitle

\centerline{\scshape Ruiyang Cai}
\medskip
{\footnotesize
 \centerline{College of Information Science and Technology, Donghua University}
  \centerline{Shanghai 201620,  China}
} 

\medskip

\centerline{\scshape Fudong Ge$^*$}
\medskip
{\footnotesize
 \centerline{School of Computer Science, China University of Geosciences}
  \centerline{Wuhan 430074,  China}
} 

\medskip
\centerline{\scshape YangQuan Chen}
\medskip
{\footnotesize
 \centerline{School of Engineering (MESA-Lab), University of California}
  \centerline{Merced, CA 95343, USA}
}
\medskip
\centerline{\scshape Chunhai Kou}
\medskip
{\footnotesize
 \centerline{Department of Applied Mathematics, Donghua University}
  \centerline{Shanghai 201620, China}
}

\bigskip

 \centerline{(Communicated by Qi L\"u)}

\begin{abstract}
This paper investigates the regional gradient controllability for ultra-slow diffusion processes governed by the time fractional diffusion systems with a Hadamard-Caputo time fractional derivative. Some necessary and sufficient conditions on regional gradient exact and approximate controllability are first given and proved in detail. Secondly, we propose an approach on how to calculate the minimum number of $\omega-$strategic actuators. Moreover, the existence, uniqueness and the concrete form of the optimal controller for the system under consideration are presented by employing the Hilbert Uniqueness Method (HUM) among all the admissible ones. Finally, we illustrate our results by an interesting example.
\end{abstract}

\section{Introduction}

\label{sec:int}

In recent years, fractional differential equations have attracted increasing interests and a wide number of monographs have been published both in its basic theory and applications (see e.g. the monographs \cite{A6,A3,A2, A1} and the references cited therein). In particular, the introduction of the continuous time random walk (CTRW) theory, which is closely related to the anomalous diffusion \cite{A7,A9}, sets off a new wave in the research on fractional order systems, whose unique ¡°power-law¡± property efficiently describes the characteristics of the anomalous diffusion.
Nevertheless, most of these studies on anomalous diffusion are involved in Riemann-Liouville and Caputo fractional derivatives, while the Hadamard fractional derivative, which was introduced in 1892 \cite{A10}, has not been researched so much as the previous two even though it has some special features. The main differences between the Hadamard fractional derivative and the above two common kinds of fractional derivative focus on two aspects. On one hand, the Hadamard fractional derivative involves an integral kernel of logarithmic function with an arbitrary exponent, which could be more effective in describing the ultra-slow diffusion processes. On the other hand, the $t\frac{d}{dt}$ in its definition has shown to be invariant on the half-axis in concerns of dilation \cite{A11}. For more knowledge of the Hadamard fractional integral and derivative, we refer the reader to \cite{A12,A16}.

The applications of the Hadamard fractional derivative are also abundant in real world. For example, fractional thermoelasticity \cite{A44}, kinetic theory of gases \cite{A43} and physical phenomena in fluctuating environments \cite{A45}, which are special cases of ultra-slow diffusion processes. The Hadamard fractional derivative can well describe these situations because of the ultra-slow increasing rate of $\log t$ compared with $t$ $(t>0)$, while both Riemann-Liouville and Caputo fractional derivatives cannot meet these needs.

Nowadays, several theoretical results on Hadamard fractional differential equations have been obtained. Two surveys \cite{A17,A18} proved the existence of solutions and weak solutions for some classes of Hadamard type fractional differential equations with and without impulse effects and presented the analytical solutions in terms of the Mittag-Leffler function. The authors in \cite{A19} investigated the initial and boundary value problems of Hadamard fractional differential equations and inclusions, and obtained some new results on them. The asymptotic behavior of solutions of nonlinear Caputo-type Hadamard fractional differential equations was established in \cite{A20}. More contributions in this field can be found in \cite{A22,A21,A25}. However, the studies on the controllability theory and applications of Hadamard time fractional diffusion systems are still on their early stages and need further investigation, even though there are a great many works on the controllability of diffusion systems \cite{A49,A50,A47}, and fractional systems \cite{A51,A52}.

Notice that in practical applications not all the diffusion systems can be controlled on the whole domain. Hence, regional controllability should be considered, where we concern the systems under consideration only within some subregions of the whole domain \cite{A26,A6}.
We claim, in addition, that the concept of regional gradient controllability makes sense in many real-life applications. One of the most important applications is in forest fire. Other than put out the fire directly, fire fighters usually try to control its spread, namely the gradient of fire intensity. Other applications can be found in thermic isolation problem, industrial ceramics, etc.
By considering the characteristics of actuators, our goal here is to study the regional gradient controllability of the Hadamard time fractional diffusion systems, which is suitable for a much wider range of physical phenomena. This is especially appealing for these systems which are not gradient controllable on the whole domain. Furthermore, it is worth mentioning that a number of remarkable new results have been derived by Zerrik et al. \cite{A27,A28,A30} on the regional gradient controllability of integer order diffusion systems. Nowadays, the authors in \cite{A31,A32,A33} have extended to study the regional (gradient) controllability of time fractional diffusion systems with the Riemann-Liouville and the Caputo fractional derivatives. Besides, we note that the initial condition of the Caputo fractional differential equation is the same as that of integer order, which has a more intuitive physical interpretation and easier to realize in engineering. With these, in this paper, we will adopt the Caputo type modification of the Hadamard fractional derivative, the so called Hadamard-Caputo fractional derivative, introduced in \cite{A21}, to establish some criteria for the regional gradient controllability of Hadamard-Caputo time fractional diffusion systems.

Motivated by the arguments above, herein, we consider the following time fractional diffusion system with a Hadamard-Caputo time fractional derivative:
\begin{equation}
\label{system}
\left\{
\begin{aligned}
\overset{}\, &
^{HC}_{a}D^{\alpha}_{t}z(x,t)=Az(x,t)+Bu(t)\mbox{ in }\Upsilon,\\
            & z(x,a)=z_{0}(x)\mbox{ in }\Omega,\\
            & z(\xi,t)=0\mbox{ on }\partial\Omega\times[a,b],\\
\end{aligned}
\right.
\end{equation}
where $\Upsilon=\Omega\times[a,b]$, $a>0$, $0<\alpha<1$, $\Omega$ is a bounded open set of $\mathbb{R}^{n}$ with smooth boundary $\partial\Omega$ and $z(x,\cdot)\in AC[a,b]\triangleq\{z(x,\cdot):\,[a,b]\rightarrow \mathbb{R}$ and $z(x,\cdot)$ is absolutely continuous in $[a,b]\}$.
Here, $A$ is the infinitesimal generator of a $C_{0}-$semigroup $\{T(t)\}$ on $L^{2}(\Omega)$ and $-A$ is a uniformly elliptic operator.
Besides, $z_{0}\in H^{1}_{0}(\Omega)$, $u$: $[a,b]\rightarrow \mathbb{R}^{m}$ and $B:\mathbb{R}^{m}\rightarrow H^{1}_{0}(\Omega)$ is a bounded linear operator.
In addition, $^{HC}_{a}D^{\alpha}_{t}$ denotes the Hadamard-Caputo time fractional derivative to be specified later.
To the best of our knowledge, no results are available on this topic and we hope that the obtained results could provide some insights into the control theory of Hadamard time fractional diffusion systems.

The rest of this paper is organized as follows. Some needed definitions and lemmas are presented in the next section. In Section 3, we give our main results, where the necessary and sufficient conditions on regional gradient exact and approximate controllability for the system \dref{system} are first explored and then we discuss the existence, uniqueness and the concrete form of the optimal controller for the system under consideration. An example is finally worked out to confirm the effectiveness of our results.

\section{Preliminary results}
  \label{sec:pre}
In this section, we recall some basic definitions and lemmas to be applied throughout this paper.

\begin{definition}\cite{A21}
\label{H}
The left-sided and right-sided Hadamard fractional integral of order $\alpha\in \mathbb{C}$, $Re(\alpha)>0$ of a function $f(t)$ are respectively, defined by
\begin{equation}
\label{def}
    _{a}^{H}I_{t}^{\alpha}f(t)\triangleq\frac{1}{\Gamma(\alpha)}\int_a^t\left(\log\frac{t}{s}\right)^{\alpha-1}f(s)\frac{ds}{s}
\end{equation}
and
\begin{equation}
    ^{H}_{t}I_{b}^{\alpha}f(t)\triangleq\frac{1}{\Gamma(\alpha)}\int_t^b\left(\log\frac{s}{t}\right)^{\alpha-1}f(s)\frac{ds}{s},
\end{equation}
where $Re(\alpha)$ denotes the real part of $\alpha$.
\end{definition}

\begin{definition}\cite{A21}
Let $Re(\alpha)\geq0$ and $n=\left[Re(\alpha)\right]+1$. If $f(t)\in AC^{n}_{\delta}[a,b]$, where $0<a<b<\infty$ and $AC^{n}_{\delta}[a,b]=\big\{f(t):[a,b]\rightarrow \mathbb{C}\mid \delta^{n-1}f(t)\in AC[a,b],\,\,\delta=t\frac{d}{dt}\big\}$.
Define the left-sided and right-sided Hadamard-Caputo fractional derivatives respectively, as follows:
\begin{equation}
    _{a}^{HC}D_{t}^{\alpha}f(t)\triangleq^{H}_{a}D_{t}^{\alpha}\left[f(t)-\sum^{n-1}_{k=0}\frac{\delta^{k}f(a)}{k!}\left(\log\frac{t}{a}\right)^{k}\right]
\end{equation}
and
\begin{equation}
    _{t}^{HC}D_{b}^{\alpha}f(t)\triangleq^{H}_{t}D_{b}^{\alpha}\left[f(t)-\sum^{n-1}_{k=0}\frac{(-\delta)^{k}f(b)}{k!}\left(\log\frac{b}{t}\right)^{k}\right],
\end{equation}
where $^{H}_{a}D^{\alpha}_{t}f(t)\,\,and\,\,^{H}_{t}D^{\alpha}_{b}f(t)$ are the left-sided and right-sided Hadamard fractional derivative, defined by
\begin{equation}
\label{HD}
^{H}_{a}D^{\alpha}_{t}f(t)\triangleq\frac{1}{\Gamma(n-\alpha)}\,\delta^{n}\int_a^t\left(\log\frac{t}{s}\right)^{n-\alpha-1}f(s)\frac{ds}{s}
\end{equation}
and
\begin{equation}
^{H}_{t}D^{\alpha}_{b}f(t)\triangleq\frac{(-1)^{n}}{\Gamma(n-\alpha)}\,\delta^{n}\int_t^b\left(\log\frac{s}{t}\right)^{n-\alpha-1}f(s)\frac{ds}{s},
\end{equation}
respectively.
In particular, if $0<Re(\alpha)<1$ and $f(t)\in AC[a,b]$, where $0<a<b<\infty$. Then $_{a}^{HC}D_{t}^{\alpha}f(t)$ exists everywhere on $[a,b]$ and can be presented by
\begin{equation}
_{a}^{HC}D_{t}^{\alpha}f(t)=\frac{1}{\Gamma(1-\alpha)}\int_a^t\left(\log\frac{t}{s}\right)^{-\alpha}f'(s)ds.
\end{equation}
\end{definition}

Let $\nabla: H^{1}_{0}(\Omega)\rightarrow \left(L^{2}(\Omega)\right)^{n}$ be the gradient operator defined by
\begin{equation*}
\nabla z\triangleq \left(\frac{\partial z}{\partial x_{1}},\ldots,\frac{\partial z}{\partial x_{n}}\right).
\end{equation*}
According to the Eq.(2) in \cite{A30}, then $\nabla^{*}:\left(L^{2}(\Omega)\right)^{n}\rightarrow H^{-1}(\Omega)$, $z\mapsto h$, the adjoint operator of $\nabla$, can be given by the unique solution of
\begin{equation}
\left\{
\begin{aligned}
\overset{}\,
& \Delta h=-div\,z\,\,\,\,\,\,\,\,in\,\,\Omega,\\
& \,h=0\,\,\,\,\,\,\,\,\,\,\,\,\,\,\,\,\,\,\,\,\,\,\,\,\,\,on\,\,\partial\Omega.\\
\end{aligned}
\right.
\end{equation}
With this, we give the following definitions.

\begin{definition}\cite{A28}
\label{2.3}
System \dref{system} is said to be regionally gradient exactly (approximately) controllable on $\omega$ at time $b$, if for every $f(x)\in \left(L^{2}(\omega)\right)^{n}$ and any $\varepsilon>0$, there exists a $u\in L^{2}\left([a,b],\,\mathbb{R}^{m}\right)$ such that
\begin{equation}
p_{\omega}\nabla z_{u}(x,b)=f(x)
\left(\parallel p_{\omega}\nabla z_{u}(x,b)-f(x)\parallel_{\left(L^{2}(\omega)\right)^{n}} <\varepsilon\right),
\end{equation}
where $p_{\omega}$ is the restriction map from $\Omega$ to its subset $\omega$.
\end{definition}

\begin{remark}
In Definition \ref{2.3}, if we choose $\omega=\Omega$, then the definition of regional controllability coincides with that of the classical controllability.
\end{remark}

\begin{definition}\cite{A27}
\label{2.4}
The actuator (actuators) is (are) said to be gradient $\omega$-strategic if system \dref{system} is regionally gradient approximately controllable on $\omega$ at time $b$.
\end{definition}

To discuss the regional controllability problem of system \dref{system}, the following generalized Mittag-Leffler function of two parameters \cite{A3} is needed:
\begin{eqnarray*}
E_{\alpha,\beta}(z)\triangleq \sum^{\infty}_{n=0}\frac{z^{n}}{\Gamma\left(n\alpha+\beta\right)},
\end{eqnarray*}
where $z\in \mathbb{C}$, $Re(\alpha)>0$. We see that when $\beta=1$, $E_{\alpha,\,\beta}(z)=E_{\alpha}(z)$, the Mittag-Leffler function in one parameter. For $\alpha=1$, we have $E_{1}(z)=e^{z}$, which is the exponential function.

Based on the Theorem 3.6 and 3.8 in \cite{A18}, we present two important lemmas as follows.

\begin{lemma}
\label{2.5}
Let $0<\alpha<1$. If there exists constants $k>-\alpha$, $l\leq0$ with $l>\max\left\{-\alpha,-\alpha-k\right\}$ and $M\geq0$ such that $\parallel Bu(t)\parallel\leq Mt^{k}(1-t)^{l}$ for all $t\in(a,b)$. Then system \dref{system} has a unique solution
\begin{equation*}
\begin{aligned}
z(x,t)=
S_{\alpha}\left(\log\frac{t}{a}\right)z_{0}(x)+\int_a^t\left(\log\frac{t}{s}\right)^{\alpha-1}K_{\alpha}\left(\log\frac{t}{s}\right)Bu(s)\frac{ds}{s},\\
\end{aligned}
\end{equation*}
where
\begin{equation*}
S_{\alpha}(t)=\int_0^{\infty} \phi_{\alpha}(s)T\left(t^{\alpha}s\right)ds
\end{equation*}
and
\begin{equation*}
K_{\alpha}(t)=\alpha\int_0^{\infty} s\phi_{\alpha}(s)T\left(t^{\alpha}s\right)ds.
\end{equation*}
Here, $\phi_{\alpha}(s)=\frac{s^{-\frac{1}{\alpha}-1}}{\alpha}\eta_{\alpha}\left(s^{-\frac{1}{\alpha}}\right)$ and $\eta_{\alpha}(s)$ is defined by
\begin{equation*}
\eta_{\alpha}(s)=-\sum_{n=1}^{\infty}\frac{\Gamma(n\alpha+1)}{n!\pi s}\left(\frac{-1}{s^{\alpha}}\right)^{n}\sin n\pi\alpha,s\in (0,\infty).
\end{equation*}
\end{lemma}

\noindent {\it{Proof.}}
Because of the equivalence between $\|\cdot\|_{L^{1}}$ and $\|\cdot\|_{L^{2}}$ in a bounded domain, thus, $u(t)\in L^{2}\left([a,b],\mathbb{R}^{m}\right)$ is bounded.
This, together with that $B$ is a bounded linear operator, yields that we can choose $k=l=0$  such that for some $M>0$, $\parallel Bu(t)\parallel\leq\parallel B\parallel\parallel u\parallel\leq M$ to satisfy the assumption in this lemma.

We choose the Picard function sequence as $\eta_{0}(x,t)=z_{0}(x)$ and
\begin{equation*}
\begin{aligned}
\eta_{i}(x,t)=z_{0}(x)+\frac{1}{\Gamma(\alpha)}
\int^{t}_{a} \left(\log\frac{t}{s}\right)^{\alpha-1}\left(A\eta_{i-1}(x,s)+Bu(s)\right)\frac{ds}{s},
\end{aligned}
\end{equation*}
for $x\in\Omega$, $t\in [a,b]$ and $i=1,2,\cdots$.
Similar to Claim 1$\sim$3 in \cite{A18}, we can prove that\\
(i) $\eta_{i}(x,t)\in C\left(\Omega\times[a,b]\right)$, $i=1,2,\cdots$;\\
(ii) $\left\{\eta_{i}(x,t)\right\}_{n\geq 1}$ converges uniformly to $\eta(x,t)$ with $(x,t)\in \Omega\times[a,b]$;\\
(iii) $\eta(x,t)=\lim\limits_{i\rightarrow\infty} \eta_{i}(x,t)$ is the unique continuous solution of the Picard function sequence given by
\begin{equation*}
\begin{aligned}
\eta (x,t)=z_{0}(x)+\frac{1}{\Gamma(\alpha)}
\int^{t}_{a} \left(\log\frac{t}{s}\right)^{\alpha-1}\left(A\eta (x,s)+Bu(s)\right)\frac{ds}{s}.
\end{aligned}\end{equation*}
Moreover, by applying the same iteration procedure as that in Theorem 3.8 of \cite{A18}, it yields that
\begin{equation*}
\begin{aligned}
\eta_{i}(x,t)= & \,\, z_{0}(x)\sum_{k=0}^{i}\frac{A^{k}\left(\log\frac{t}{a}\right)^{k\alpha}}{\Gamma(k\alpha+1)}+\frac{1}{\Gamma(\alpha)}
\int^{t}_{a} \left(\log\frac{t}{s}\right)^{\alpha-1}\mu_{i}(s)Bu(s)\frac{ds}{s}\\
\rightarrow & \,\, S_{\alpha}\left(\log\frac{t}{a}\right)z_{0}(x)+\int_a^t\left(\log\frac{t}{s}\right)^{\alpha-1}K_{\alpha}\left(\log\frac{t}{s}\right)Bu(s)\frac{ds}{s}
\end{aligned}
\end{equation*}
as $i\rightarrow\infty$, where $\mu_{i}(s)=\sum\limits_{k=0}^{i}\frac{A^{k}\left(\log\frac{t}{s}\right)^{k\alpha}}{\Gamma((k+1)\alpha)}.$
$\,\,\,\,\,\,\,\,\,\,\,\,\,\,\,\,\,\,\,\,\,\,\,\,\,\,\,\,\,\,\,\,\,\,\,\,\,\,\,\,\,\,\,\,\,\,\,\,\,\,\,\,\,\,\,\,\,\,\,\,\,\,\,\,\,\,\,\,\,\,\,\,\,\,\,\,\,\,\,\,\,\,\,\,\,\,\,\,\,\,\,\,\,\,\,\,\,\,\,\,\,\,\,\,\,\,\square$

\begin{remark}
For the abstract operators $A$ and $B$ involved in system \dref{system}, the last formula is given in terms of the operator $S_{\alpha}(t)$ and $K_{\alpha}(t)$,
because of the uniqueness of the solution for a linear system (see e.g. \cite{A3} or the Theorem 3.8 of \cite{A18}).
Moreover, the uniqueness of the solution shows that when $A$ is a real number, $S_{\alpha}(t)=E_{\alpha}\left(At^{\alpha}\right)$,   $K_{\alpha}(t)=E_{\alpha,\alpha}\left(At^{\alpha}\right)$ and the solution coincides with that in \cite{A18}.
For more arguments and properties of $S_{\alpha}(t)$, $K_{\alpha}(t)$, $\phi_{\alpha}(t)$ and $\eta_{\alpha}(t)$, we refer the reader to \cite{A37,A38,A39}.
\end{remark}

\begin{lemma}
\label{2.6}
Let $0<\alpha<1$. Then the following problem
\begin{equation*}
\left\{
\begin{aligned}
\overset{}\, &
^{H}_{a}D^{\alpha}_{t}z(x,t)=Az(x,t)\mbox{ in }\Upsilon,\\
            & \lim_{t\rightarrow a} \,^{H}_{a}D^{\alpha-1}_{t}z(x,t)=z_{0}(x)\mbox{ in }\Omega,\\
            & z(\xi,t)=0\mbox{ on }\partial\Omega\times[a,b]\\
\end{aligned}
\right.
\end{equation*}
has a unique solution satisfying
\begin{equation*}
z(x,t)=
\left(\log\frac{t}{a}\right)^{\alpha-1}K_{\alpha}\left(\log\frac{t}{s}\right)z_{0}(x).
\end{equation*}
\end{lemma}

The proof is similar to that of Lemma \ref{2.5} and we omit it.

%

\begin{lemma}\cite{A35}
Let $\Omega\subseteq \mathbb{R}^{n}$ be a bounded open set and $C^{\infty}_{0}(\Omega)$ be the set of infinitely differentiable functions on $\Omega$ with compact support. If for $v_{1}\in L^{1}(\Omega)$, such that $$\int_{\Omega} v_{1}(x)v_{2}(x)dx=0,\,\,\,\,\forall v_{2}\in C^{\infty}_{0}(\Omega).$$
Then $v_{1}=0,\,\,a.e.\,\,in\,\,\Omega.$
\end{lemma}

\begin{lemma}
\label{2.8}
Define $Q$ as $Qf(t)=f\left(\frac{ab}{t}\right)$ and assume $0<\alpha<1$. Then we have the following equalities:
\begin{equation*}
\begin{aligned}
(i) \,\,Q^{H}_{a}I^{\alpha}_{t}f(t)=^{H}_{t}I^{\alpha}_{b}Qf(t), \,\, \,\,
(ii)\,\,Q^{H}_{a}D^{\alpha}_{t}f(t)=^{H}_{t}D^{\alpha}_{b}Qf(t),\\
(iii) \,\,^{H}_{a}I^{\alpha}_{t}Qf(t)=Q^{H}_{t}I^{\alpha}_{b}f(t), \,\, \,\,
(iv)\,\,^{H}_{a}D^{\alpha}_{t}Qf(t)=Q^{H}_{t}D^{\alpha}_{b}f(t).\\
\end{aligned}
\end{equation*}
\end{lemma}

\noindent {\it{Proof.}} $(i)$ From Definition \ref{H}, we have
\begin{equation*}
\begin{aligned}
^{H}_{t}I^{\alpha}_{b}Qf(t)
&=^{H}_{t}I^{\alpha}_{b}f\left(\frac{ab}{t}\right)\\
&=\frac{1}{\Gamma(\alpha)}\int_t^b\left(\log\frac{s}{t}\right)^{\alpha-1}f\left(\frac{ab}{s}\right)\frac{ds}{s} \\
&=\frac{1}{\Gamma(\alpha)}\int_a^\frac{ab}{t}\left(\log\frac{ab}{st}\right)^{\alpha-1}f(s)\frac{ds}{s} \\
&=Q^{H}_{a}I^{\alpha}_{t}f(t).\\
\end{aligned}
\end{equation*}

$(iv)$ It follows from \dref{HD} that
\begin{equation*}
\begin{aligned}
~&^{H}_{a}D^{\alpha}_{t}Qf(t)=^{H}_{a}D^{\alpha}_{t}f\left(\frac{ab}{t}\right)\\
=&\frac{t}{\Gamma(1-\alpha)}\,\frac{d}{dt}\int_a^t\left(\log\frac{t}{s}\right)^{-\alpha}f\left(\frac{ab}{s}\right)\frac{ds}{s}\\
=&\frac{t}{\Gamma(1-\alpha)}\,\frac{d}{dt}\int_\frac{ab}{t}^b \left(\log\frac{ut}{ab}\right)^{-\alpha}f(u)\frac{du}{u}.\\
\end{aligned}
\end{equation*}
On the other hand,
\begin{equation*}
\begin{aligned}
 Q\left(^{H}_{t}D^{\alpha}_{b}f(t)\right)
=&Q\left(\frac{-t}{\Gamma(1-\alpha)}\,\frac{d}{dt}\int_t^b\left(\log\frac{s}{t}\right)^{-\alpha}f(s)\frac{ds}{s}\right)
\end{aligned}
\end{equation*}\begin{equation*}
\begin{aligned}=&\frac{-ab}{t\Gamma(1-\alpha)}\,\frac{d}{d\left(\frac{ab}{t}\right)}\int_\frac{ab}{t}^b\left(\log\frac{st}{ab}\right)^{-\alpha}f(s)\frac{ds}{s}\\
=&\frac{-ab}{t\Gamma(1-\alpha)}\frac{-t^{2}}{ab}\,\frac{d}{dt}\int_\frac{ab}{t}^b\left(\log\frac{st}{ab}\right)^{-\alpha}f(s)\frac{ds}{s}\\
=&\frac{t}{\Gamma(1-\alpha)}\,\frac{d}{dt}\int_\frac{ab}{t}^b \left(\log\frac{st}{ab}\right)^{-\alpha}f(s)\frac{ds}{s}.
\end{aligned}
\end{equation*}
Hence, $^{H}_{a}D^{\alpha}_{t}Qf(t)=Q^{H}_{t}D^{\alpha}_{b}f(t)$. Based on these,
$(ii)$ and $(iii)$ can be proved similarly.
$\,\,\,\,\,\,\,\,\,\,\,\,\,\,\,\,\,\,\,\,\,\,\,\,\,\,\,\,\,\,\,\,\,\,\,\,\,\,\,\,\,\,\,\,\,\,\,\,\,\,\,\,\,\,\,\,\,\,\,\,\,\,\,\,\,\,\,\,\,\,\,\,\,\,\,\,\,\,\,\,\,\,\,\,\,\,\,\,\,\,\,\,\,\,\,\,\,\,\,\,\,\,\,\,\,\,\,\,\,\,\,\,\,\,\,\,\,\,\,\,\,\,\,\,\,\,\,\,\,\,\,\,\,\,\,\,\,\,\,\,\,\,\,\,\,\,\,\,\,\,\,\,\,\,\,\,\,\,\,\,\,\,\,\,\,\,\,\,\,\,\,\,\,\,\,\,\,\,\,\,\,\,\,\,\,\,\,\,\,\,\,\,\,\,\,\,\,\square$

\begin{lemma}\cite{A34}
\label{2.9}
Suppose that $X,\,Y,\,Z$ are reflexive Hilbert spaces, $f\in L(X,Z)$ and $g\in L(Y,Z)$. Then
$$Im(f)\subseteq Im(g)$$
is equivalent to
$$\exists K>0,\,\,s.t.\,\parallel f^{*}z\parallel_{X^{*}}\leq K\parallel g^{*}z\parallel_{Y^{*}},\,\,\forall z\in Z.$$
\end{lemma}

\begin{lemma}\cite{A36}
\label{2.10}
Let $U$ be a closed, convex subset of a Hilbert space $H$. Assume that for $v\in U$, the functional $v\rightarrow J(v)$ is strict convex, differentiable and satisfies $J(v)\rightarrow +\infty$ as $\parallel v\parallel\rightarrow +\infty$. Then the unique element $u$ in $U$ satisfying $J(u)=\inf_{v\in U}J(v)$ is characterized by $$J'(u)(v-u)\geq0,\,\,\,\,\forall v\in U.$$
\end{lemma}

\section{Main results}
  \label{sec:res}

By Proposition 3.1 in \cite{A40}, since system \dref{system} is a linear system, without loss of generality, we can suppose $z_{0}(x)=0$. Then, the solution of system \dref{system} reduces to
\begin{equation}
z(x,t)=\int_a^t\left(\log\frac{t}{s}\right)^{\alpha-1}K_{\alpha}\left(\log\frac{t}{s}\right)Bu(s)\frac{ds}{s}.
\end{equation}

Define the operator $H:\,L^{2}\left([a,b],\mathbb{R}^{m}\right)\rightarrow L^{2}(\Omega)$,
\begin{equation}
\begin{aligned}
Hu
\triangleq&\int^{b}_{a}\left(\log\frac{b}{s}\right)^{\alpha-1}K_{\alpha}\left(\log\frac{b}{s}\right)Bu(s)\frac{ds}{s}\\ =&\int^{\log\frac{b}{a}}_{0}s^{\alpha-1}K_{\alpha}(s)Bu(be^{-s})ds£¬ \\
\end{aligned}
\end{equation}
for any $u(t)\in L^{2}\left([a,b],\,\mathbb{}R^{m}\right)$.
Since $B$ is bounded and linear,
let $H^{*}$, be the adjoint operator of $H$, that is, $\langle Hu,\,v\rangle=\langle u,\,H^{*}v\rangle$ for any $v\in L^{2}(\Omega)$.
For any $v\in L^{2}(\Omega)$, one has
\begin{equation*}
\begin{aligned}
\langle Hu,\,v\rangle
&=\left\langle\int^{b}_{a}\left(\log\frac{b}{s}\right)^{\alpha-1}K_{\alpha}\left(\log\frac{b}{s}\right)Bu(s)\frac{ds}{s},\,v\right\rangle \\
&= \int^{b}_{a}\left\langle\left(\log\frac{b}{s}\right)^{\alpha-1}K_{\alpha}\left(\log\frac{b}{s}\right)\frac{Bu(s)}{s},\,v\right\rangle ds \\
&= \int^{b}_{a} \left\langle u(s),\,\left(\log\frac{b}{s}\right)^{\alpha-1}K^{*}_{\alpha}\left(\log\frac{b}{s}\right)\frac{B^{*}v}{s}\right\rangle ds.\\
\end{aligned}
\end{equation*}
Hence,
\begin{equation}
\label{13}
H^{*}v=\frac{1}{t}B^{*}\left(\log\frac{b}{t}\right)^{\alpha-1}K_{\alpha}^{*}\left(\log\frac{b}{t}\right)v,\,\,\forall v\in L^{2}(\Omega).
\end{equation}
In addition, the adjoint operator of the restriction map $p_{\omega}$ is defined by
\begin{equation}
p_{\omega}^{*}f(x)=\left\{
\begin{aligned}
\overset{}\, &
f(x),\,\,\,\,\,\,\,x\in\omega,\\
            & 0,\,\,\,\,\,\,\,\,\,\,\,\,\,\,\,x\in\Omega\setminus\omega.\\
\end{aligned}
\right.
\end{equation}
When $n=1$, we denote $p_{\omega}$ as $p_{1,\omega}$, so does $p^{*}_{1,\omega}$.

We can deduce the following results.

\begin{theorem}
\label{3.1}
The following statements are equivalent:\\
$(i)$ System \dref{system} is regionally gradient exactly controllable on $\omega$ at time $b$;\\
$(ii)$ $Im\left(p_{\omega}\nabla H\right)=\left(L^{2}(\omega)\right)^{n}$;\\
$(iii)$ $Ker\left(p_{\omega}\right)+Im(\nabla H)=\left(L^{2}(\Omega)\right)^{n}$;\\
$(iv)$ There exists $K>0$, such that $\forall z\in \left(L^{2}(\omega)\right)^{n}$,
$$\parallel z\parallel_{\left(L^{2}(\omega)\right)^{n}}\leq K\parallel H^{*}\nabla^{*}p_{\omega}^{*}z\parallel_{L^{2}\left(a,b;R^{m}\right)}.$$
\end{theorem}

\begin{proof}
The equivalence between $(i)$ and $(ii)$ can be easily obtained from Definition \ref{2.3}. In Lemma \ref{2.9}, choose $X=Z=\left(L^{2}(\omega)\right)^{n}$, $Y=\left(L^{2}(\Omega)\right)^{n}$, $f=I$, the identity operator and $g=p_{\omega}\nabla H$, we have $(i)$$\Leftrightarrow$$(iv)$. So we only need to prove $(ii)$$\Leftrightarrow$$(iii)$.

$(ii)$$\Rightarrow$$(iii)$: $\forall x_{1}\in Ker\left(p_{\omega}\right)$, $x_{2}\in Im(\nabla H)$, since
$$p_{\omega}\left(x_{1}+x_{2}\right)=p_{\omega}x_{2}\in Im\left(p_{\omega}\nabla H\right)=\left(L^{2}(\omega)\right)^{n},$$
it follows that $x_{1}+x_{2}\in \left(L^{2}(\Omega)\right)^{n}$, that is,
\begin{equation}
\label{15}
Ker\left(p_{\omega}\right)+Im(\nabla H)\subseteq \left(L^{2}(\Omega)\right)^{n}.
\end{equation}

Next, $\forall x\in \left(L^{2}(\Omega)\right)^{n}$, define $\widetilde{x}\triangleq p_{\omega}x\in \left(L^{2}(\omega)\right)^{n}$ $=Im\left(p_{\omega}\nabla H\right)$, and $y\triangleq x-\widetilde{x}$.
So, there exists a $u\in L^{2}(\Omega)$, such that $p_{\omega}\nabla Hu=\widetilde{x}$. Then $\widetilde{x}\in Im(\nabla H)$. Besides,  $$p_{\omega}y=p_{\omega}\left(x-\widetilde{x}\right)=p_{\omega}x-p_{\omega}\widetilde{x}=0,$$ that is, $y\in Ker\left(p_{\omega}\right)$.
Hence,
\begin{equation}
\left(L^{2}(\Omega)\right)^{n}\subseteq Ker\left(p_{\omega}\right)+Im(\nabla H).
\end{equation}

$(iii)$$\Rightarrow$$(ii)$: From $(iii)$, for any $x\in \left(L^{2}(\omega)\right)^{n}\subseteq \left(L^{2}(\Omega)\right)^{n}$, there exists $x_{1}\in Ker\left(p_{\omega}\right)$ and $x_{2}\in Im(\nabla H)$ such that $x=x_{1}+x_{2}$.
Then a $u\in L^{2}(\Omega)$ can be found satisfying $\nabla Hu=x_{2}$.
Hence, $$x=p_{\omega}x=p_{\omega}\left(x_{1}+\nabla Hu\right)=p_{\omega}\nabla Hu\in Im\left(p_{\omega}\nabla H\right),$$
that is,
\begin{equation}
\left(L^{2}(\Omega)\right)^{n}\subseteq Im\left(p_{\omega}\nabla H\right).
\end{equation}
Next, $\forall x\in Im\left(p_{\omega}\nabla H\right)$, together with the definition of $H$, we immediately get that there exists a $u\in L^{2}(\Omega)$, such that $p_{\omega}\nabla Hu=x\in \left(L^{2}(\omega)\right)^{n}$, namely
\begin{equation}
\label{18}
Im\left(p_{\omega}\nabla H\right)\subseteq \left(L^{2}(\omega)\right)^{n}.
\end{equation}
Combining \dref{15}$\thicksim$\dref{18}, we get that $(ii)$$\Leftrightarrow$$(iii)$ and complete the proof.
\end{proof}

\begin{theorem}
\label{3.2}
The following statements are equivalent:\\
$(i)$ System \dref{system} is regionally gradient approximately controllable on $\omega$ at time $b$;\\
$(ii)$ $\overline{Im\left(p_{\omega}\nabla H\right)}=\left(L^{2}(\omega)\right)^{n}$;\\
$(iii)$ $Ker\left(p_{\omega}\right)+\overline{Im(\nabla H)}=\left(L^{2}(\Omega)\right)^{n}$;\\
$(iv)$ $p_{\omega}\nabla HH^{*}\nabla^{*}p_{\omega}^{*}$ is a positive definite operator;\\
$(v)$ If $\langle p_{\omega}\nabla Hu,z\rangle=0,\,\forall u\in L^{2}(\Omega)$, leads to $z=0$, where $\langle\cdot,\cdot\rangle$ denotes the inner product.
\end{theorem}

\begin{proof}
Similar to Theorem \ref{3.1}, one has $(i)$$\Leftrightarrow$$(ii)$$\Leftrightarrow$$(iii)$. Then we prove $(ii)$$\Leftrightarrow$$(iv)$ and $(ii)$$\Leftrightarrow$$(v)$.

$(ii)$$\Leftrightarrow$$(iv)$: For any $y$, $z\in \left(L^{2}(\omega)\right)^{n}$,
$$\langle p_{\omega}\nabla HH^{*}\nabla^{*} p_{\omega}^{*}y,\,z\rangle=\langle y,\,p_{\omega}\nabla HH^{*}\nabla^{*} p_{\omega}^{*}z\rangle$$
and $\langle p_{\omega}\nabla HH^{*}\nabla^{*} p_{\omega}^{*}y,\,y\rangle=\langle H^{*}\nabla^{*} p_{\omega}^{*}y,\,H^{*}\nabla^{*} p_{\omega}^{*}y\rangle$. These, together with the equivalence between $\overline{Im\left(p_{\omega}\nabla H\right)}=\left(L^{2}(\omega)\right)^{n}$ and the domain of
$p_{\omega}\nabla HH^{*}\nabla^{*}p_{\omega}^{*}$ is dense in $\left(L^{2}(\omega)\right)^{n}$, lead to the result.

$(ii)$$\Leftrightarrow$$(v)$: We can easily see that $(ii)$$\Rightarrow$$(v)$ and $\overline{Im\left(p_{\omega}\nabla H\right)}\subseteq \left(L^{2}(\omega)\right)^{n}$.
Moreover, we claim that $\left(L^{2}(\omega)\right)^{n}\subseteq\overline{Im\left(p_{\omega}\nabla H\right)}$. If not, there is a nonzero $z\in \left(L^{2}(\omega)\right)^{n}\setminus\overline{Im\left(p_{\omega}\nabla H\right)}$, such that
\begin{equation*}
\langle p_{\omega}\nabla Hu,\,z\rangle=0,\,\,\forall u\in L^{2}(\Omega),
\end{equation*}
which leads to a contradiction.
\end{proof}

Next, we will take research on the description of the actuators and give the minimum number of the actuators to guarantee the desired performance.
By \cite{A41}, we see that an actuator can be characterized by $\left(P,d\right)$, where $P\subseteq\Omega$ represents the working area of the actuator and $d$ is its spatial distribution.
Here, we suppose that the system is controlled by $m$ actuators with corresponding characteristics $\left(P_{i},d_{i}(x)\right)$, where $d_{i}(x)\in L^{2}(\Omega)$, $i=1,2,\ldots,m$. Let
\begin{equation}
Bu=\sum^{m}_{i=1}\chi_{P_{i}}d_{i}(x)u_{i}(t),
\end{equation}
where $u=\left(u_{1},\ldots,u_{m}\right)$, $u_{i}\in L^{2}\left([a,b],\mathbb{R}\right)$ and $\chi_{P_{i}}$ denotes the indicator function on $P_{i}$.
Then, system \dref{system} is equivalent to
\begin{equation}
\left\{
\begin{aligned}
\overset{}\, &
^{HC}_{a}D^{\alpha}_{t}z(x,t)=Az(x,t)+\sum^{m}_{i=1}\chi_{P_{i}}d_{i}(x)u_{i}(t),\\
            & z(x,a)=0 \mbox{ in }\Omega,\\
            & z(\xi,t)=0\mbox{ on }\partial\Omega\times[a,b],\\
\end{aligned}
\right.
\end{equation}
whose solution is given by
\begin{equation}
\label{21}
\begin{aligned}
z(x,t)
=&\int_a^t\left(\log\frac{t}{s}\right)^{\alpha-1}K_{\alpha}\left(\log \frac{t}{a}\right)
\sum^{m}_{i=1}\chi_{P_{i}}d_{i}(x)u_{i}(s)\frac{ds}{s} \\
=&\int_0^{\log\frac{t}{a}}s^{\alpha-1}K_{\alpha}(s)\chi_{P_{i}}d_{i}(x)u_{i}(te^{-s})ds.\\
\end{aligned}
\end{equation}

Suppose that $\lambda_{1},\ldots,\lambda_{k},\ldots$ are the eigenvalues of $-A$ with corresponding multiplicities $r_{1},\ldots,\,r_{k},\ldots$, satisfying $0<\lambda_{1}<\ldots<\lambda_{k}<\ldots$, and $\lim_{k\rightarrow \infty}\lambda_{k}=\infty$.
The orthonormal eigenfunctions $\alpha_{kj}(x)$, $j=1,\ldots,r_{k}$ corresponding to $\lambda_{k}$, for $k=1,2,\ldots$ consist of an orthonormal basis of $L^{2}(\Omega)$.
With this, for any $z\in L^{2}(\Omega)$, the $C_{0}-$semigroup $\{T(t)\}_{t\geq0}$ can be shown as
\begin{equation*}
T(t)z=\sum^{\infty}_{k=1}\sum^{r_{k}}_{j=1}e^{-\lambda_{j}t}\langle z,\,\alpha_{kj}(x)\rangle\alpha_{kj}(x),
\end{equation*}
that is
\begin{equation}
\label{22}
z=\sum^{\infty}_{k=1}\sum^{r_{k}}_{j=1}\langle z,\alpha_{kj}(x)\rangle\alpha_{kj}(x).
\end{equation}
From \dref{22}, the solution of system \dref{system} in \dref{21} can be expressed by
\begin{equation}
\label{23}
\begin{aligned}
z(x,t)=\sum^{\infty}_{k=1}\sum^{r_{k}}_{j=1}\sum^{m}_{i=1}&\int_0^{\log\frac{t}{a}}s^{\alpha-1}E_{\alpha,\,\alpha}\left(-\lambda_{k}s^{\alpha}\right)
 d^{i}_{kj}(x)u_{i}\left(te^{-s}\right)ds\cdot\alpha_{kj}(x),
\end{aligned}
\end{equation}
where $d^{i}_{kj}(x)\triangleq\langle\chi_{P_{i}}d_{i}(x),\alpha_{kj}(x)\rangle$.

\begin{theorem}
\label{3.3}
Define
\begin{equation*}
D_{k}^{l}
\triangleq\frac{\partial}{\partial x_{l}}\begin{bmatrix}
d^{1}_{k1}(x)\alpha_{k1}(x)  & \cdots\ & d^{1}_{kr_{k}}(x)\alpha_{kr_{k}}(x)\\
 \vdots   & \cdots  & \vdots  \\
d^{m}_{k1}(x)\alpha_{k1}(x)  & \cdots\ & d^{m}_{kr_{k}}(x)\alpha_{kr_{k}}(x)\\
\end{bmatrix}.
\end{equation*}
Then $\left(P_{i},d_{i}(x)\right),\,i=1,\ldots,m$ are gradient $\omega$-strategic if and only if for any $z\in \left(L^{2}(\omega)\right)^{n}$,
\begin{equation}
\label{24}
\sum^{\infty}_{k=1}\sum^{r_{k}}_{j=1}t^{\alpha-1}E_{\alpha,\,\alpha}\left(-\lambda_{k}t^{\alpha}\right)\sum^{n}_{l=1}D^{l}_{k}\Phi_{kl}=0\,\,\Rightarrow\,\,z=0,
\end{equation}
where $z=\left(z_{1},\ldots,z_{n}\right)^{\top}$, $\Phi_{kl}\triangleq\left(z_{k1l},\ldots,z_{kr_{k}l}\right)^{\top}$ and $z_{kjl}\triangleq\left\langle p^{*}_{1,\omega}z_{l},\,\alpha_{kj}(x)\right\rangle$.
When $n=1$, \dref{24} is equivalent to
$$m\geq r\triangleq \sup\left\{r_{k}\right\}\,\,and\,\,rank\,D^{1}_{k}=r_{k},$$ for $k=1,2,\ldots$.
\end{theorem}

\begin{proof}
According to Definition \ref{2.4} and Theorem \ref{3.2}, we know $\left(P_{i},d_{i}(x)\right),\,i=1,\ldots,m$ are gradient $\omega$-strategic is equivalent to $z=0$, provided that $\langle p_{\omega}\nabla Hu,z\rangle=0$, for any $u\in L^{2}\left([a,b],\mathbb{R}^{m}\right)$.

From $\langle p_{\omega}\nabla Hu,z\rangle=0$ and \dref{23}, we have
\begin{equation}
\label{25}
\begin{aligned}
\sum^{n}_{l=1}&{\langle}\sum^{\infty}_{k=1}\sum^{r_{k}}_{j=1}\sum^{m}_{i=1}\int_0^{\log\frac{b}{a}}s^{\alpha-1}E_{\alpha,\,\alpha}\left(-\lambda_{k}s^{\alpha}\right)\times \\
&u_{i}\left(be^{-s}\right)ds\cdot\frac{\partial}{\partial x_{l}}\left(d^{i}_{kj}(x)\alpha_{kj}(x)\right),\,p_{1\omega}^{*}z_{l}{\rangle}=0.\\
\end{aligned}
\end{equation}
Then, the arbitrariness of $u$ leads to the equivalence between \dref{25} and
\begin{equation*}
\sum^{\infty}_{k=1}\sum^{r_{k}}_{j=1}t^{\alpha-1}E_{\alpha,\,\alpha}\left(-\lambda_{k}t^{\alpha}\right)\sum^{n}_{l=1}D^{l}_{k}\Phi_{kl}=0
\end{equation*}
and the proof of the first part is completed.

When $n=1$, since $t^{\alpha-1}E_{\alpha,\alpha}\left(-\lambda_{k}t^{\alpha}\right)>0$, for any $t\in[a,b]$, the equivalence between \dref{24} and
\begin{equation*}
m\geq \sup\left\{r_{k}\right\}\,\,and\,\,rank\,D^{1}_{k}=r_{k},\,\,\,\,\,\,k=1,2,\ldots,
\end{equation*}
can be easily derived by the knowledge of linear algebra.
\end{proof}

\begin{remark}
\label{3}
If for every $k$, $\lambda_{k}$ is a single eigenvalue of $-A$, Theorem \ref{3.3} shows that we can steer system \dref{system} to be regionally gradient approximately controllable by one actuator; and if there is a $\widetilde{k}$, such that the multiplicity of $\lambda_{\widetilde{k}}$ is infinite, then the number of actuators must be infinite.
\end{remark}

The last part in this section aims to provide a method to find out the optimal actuators with minimum energy to achieve the regional gradient approximate controllability among all admissible ones. The HUM, which was first introduced by Lions \cite{A36,A42}, is the main method to be used.

For any given target state $f\in \left(L^{2}(\omega)\right)^{n}$, define
\begin{equation}
U_{ad}\triangleq \left\{u\in L^{2}(\Omega)\mid p_{\omega}\nabla Hu=f(x)\right\},
\end{equation}
and the minimum energy (cost) functional
\begin{equation}
\label{27}
\inf_{u\in U_{ad}} J(u)\triangleq \inf_{u\in U_{ad}} \int^{b}_{a}\parallel u(t)\parallel_{\mathbb{R}^{m}}^{2}dt.
\end{equation}
When system \dref{system} is regionally gradient controllable, we can easily see that $U_{ad}$ is nonempty.

It's also worth mentioning that the cost for regional gradient controllability is not more than that for regional controllability. Let $\widetilde{U}_{ad}$ be the admissible control set for the corresponding regional controllability. Since $Hu=f(x)$ implies $\nabla Hu=\nabla f(x)$, thus $\widetilde{U}_{ad}\subseteq U_{ad}$ and
\begin{equation*}
\inf_{u\in U_{ad}} J(u)\leq\inf_{u\in \widetilde{U}_{ad}} J(u).
\end{equation*}

Next, by using the HUM, we'll provide the unique solution to the minimum energy functional \dref{27} to guarantee the regional gradient approximate controllability.

Define $G\triangleq\left\{g\in \left(L^{2}(\Omega)\right)^{n}\mid g=0\,\,in\,\,\Omega\setminus\omega\,\,and\,\,\exists\,!\,\,\widetilde{g}\in H^{1}_{0}(\Omega),\,\,s.t.\,\,\nabla\widetilde{g}=g\right\}$. \\
Then, $\widetilde{g}=\nabla^{*}p_{\omega}^{*}g\in H^{1}_{0}(\Omega)$, for any $g\in G$.
Consider
\begin{equation}
\label{28}
\left\{
\begin{aligned}
\overset{}\, &
Q^{H}_{t}D^{\alpha}_{b}\varphi(x,t)=A^{*}Q\varphi(x,t)\mbox{ in }\Upsilon,\\
            & \lim_{t\rightarrow b} Q^{H}_{t}D^{\alpha-1}_{b}\varphi(x,t)=\nabla^{*}p^{*}_{\omega}g(x)\mbox{ in }\Omega,\\
            & \varphi(\xi,b-t)=0\mbox{ on }\partial\Omega\times[a,b].\\
\end{aligned}
\right.
\end{equation}
According to Lemma \ref{2.8} and then utilizing Lemma \ref{2.6}, the unique solution of system \dref{28} satisfies:
\begin{equation}
\varphi(x,t)=\left(\log\frac{b}{t}\right)^{\alpha-1}K_{\alpha}^{*}\left(\log\frac{b}{t}\right)\nabla^{*}p^{*}_{\omega}g(x).
\end{equation}
Define
\begin{equation}
\label{30}
\parallel g\parallel_{G}^{2}\triangleq\int^{b}_{a}\left\|\frac{1}{t}B^{*}K(t)\nabla^{*}p^{*}_{\omega}g(x)\right\|^{2}dt,
\end{equation}
where $K(t)=\left(\log\frac{b}{t}\right)^{\alpha-1}K_{\alpha}^{*}\left(\log\frac{b}{t}\right)$, then we have the following lemma.

\begin{lemma}
\label{3.4}
If system \dref{system} is regionally gradient approximately controllable, then \dref{30} is a norm on $G$.
\end{lemma}

\begin{proof} It's obvious that \dref{30} defines a semi-norm on $G$.

From $\parallel g\parallel_{G}=0$, we have
$$\frac{1}{t}B^{*}K(t)\nabla^{*}p^{*}_{\omega}g(x)=0.$$

If system \dref{system} is regionally gradient approximately controllable, from Theorem \ref{3.2}, we know $\overline{Im\left(p_{\omega}\nabla H\right)}=\left(L^{2}(\omega)\right)^{n}$. Hence, $$Ker\left(H^{*}\nabla^{*}p^{*}_{\omega}\right)=\{0\},$$
that is, $\parallel g\parallel_{G}=0$ leads to $g=0$. Thus, \dref{30} defines a norm on $G$.
\end{proof}

Now, consider the following two systems
\begin{equation}
\left\{
\begin{aligned}
\overset{}\, &
^{HC}_{a}D^{\alpha}_{t}\Psi_{1}(x,t)=A\Psi_{1}(x,t)+\frac{1}{t}BB^{*}\varphi(x,t)\mbox{ in }\Upsilon,\\
            & \Psi_{1}(x,a)=0\mbox{ in }\Omega,\\
            & \Psi_{1}(\xi,t)=0\mbox{ on }\partial\Omega\times[a,b]\\
\end{aligned}
\right.
\end{equation}
and
\begin{equation}
\left\{
\begin{aligned}
\overset{}\, &
^{HC}_{a}D^{\alpha}_{t}\Psi_{2}(x,t)=A\Psi_{2}(x,t)\mbox{ in }\Upsilon,\\
            & \Psi_{2}(a)=y_{0}(x)\mbox{ in }\Omega,\\
            & \Psi_{2}(\xi,t)=0\mbox{ on }\partial\Omega\times[a,b].\\
\end{aligned}
\right.
\end{equation}
Define $Fg\triangleq p_{\omega}\nabla\Psi_{1}(x,b)$. According to the superposition principle in linear system, we can provide the regional gradient approximate controllability of system \dref{system} if we can solve
\begin{equation}
\label{33}
Fg=f(x)-p_{\omega}\nabla\Psi_{2}(x,b).
\end{equation}

\begin{theorem}
\label{3.5}
If system \dref{system} is regionally gradient approximately controllable, then for any given $f(x)\in \left(L^{2}(\omega)\right)^{n}$, \dref{33} exists a unique solution $g\in G$, and the actuator $u^{*}(t)=\frac{1}{t}B^{*}\varphi(x,t)$ is the unique solution of \dref{27}.
\end{theorem}

\begin{proof} Given any $g\in G$, it yields that
\begin{equation*}
\begin{aligned}
& \left\langle g,Fg\right\rangle=\left\langle g,p_{\omega}\nabla\Psi_{1}(x,b)\right\rangle=\left\langle \nabla^{*}p_{\omega}^{*}g,\Psi_{1}(x,b)\right\rangle \\
=&\langle \nabla^{*}p_{\omega}^{*}g,\int^{b}_{a}\left(\log\frac{b}{s}\right)^{\alpha-1}K_{\alpha}\left(\log\frac{b}{s}\right)BB^{*}\varphi(x,s)\frac{ds}{s^{2}}\rangle \\
= & \int^{b}_{a} \left\|\frac{1}{t}B^{*}\varphi(x,t)\right\|^{2}dt=\left\|g\right\|^{2}_{G}.
\end{aligned}
\end{equation*}
Hence, \dref{33} exists a unique solution $\widehat{g}$. From $u^{*}(t)=\frac{1}{t}B^{*}\varphi(x,t)$, we can easily check that $p_{\omega}\nabla Hu^{*}=f(x)$, that is, $u^{*}\in U_{ad}$.

For any $\widetilde{u}\in L^{2}(\Omega)$, satisfying $p_{\omega}\nabla H\widetilde{u}=f(x)$, we have $p_{\omega}\nabla Hu^{*}=p_{\omega}\nabla H\widetilde{u}.$
Therefore, it leads to
\begin{equation*}
\begin{aligned}
& J'(u^{*})\left(u^{*}-\widetilde{u}\right)
=2\int^{b}_{a}\left\langle u^{*}(s),u^{*}(s)-\widetilde{u}(s)\right\rangle ds\\
= & 2\int^{b}_{a}\left\langle\frac{1}{s}B^{*}\varphi(x,s),\,u^{*}(s)-\widetilde{u}(s)\right\rangle ds\\
= & 2\left\langle\nabla^{*}p_{\omega}^{*}g,\,\int^{b}_{a} T(s)\frac{ds}{s}\right\rangle\\
= & 2\left\langle g,\,p_{\omega}\nabla Hu^{*}-p_{\omega}\nabla H\widetilde{u}\right\rangle=0,
\end{aligned}
\end{equation*}
where
\begin{equation*}
T(s)=\left(\log\frac{b}{s}\right)^{\alpha-1}K_{\alpha}\left(\log\frac{b}{s}\right)B\left(u^{*}(s)-\widetilde{u}(s)\right).
\end{equation*}
By applying Lemma \ref{2.10}, we conclude that $u^{*}$ is the unique solution of \dref{27}.
\end{proof}

\begin{remark}
If we choose $f(x)=0$ in Definition \ref{2.3}, then the regional gradient exact/approximate null controllability can be guaranteed.
\end{remark}


\section{An example}
  \label{sec:exa}
Let $\Omega=[-1,1]\times[-1,1]$ and consider the following diffusion system with one zone actuator:
\begin{equation}
\label{eg}
\left\{
\begin{aligned}
\overset{}\, &
^{HC}_{2}D^{0.5}_{t}z(x,t)=-\triangle z(x,t)+\chi_{P}u(t)\mbox{ in }\widetilde{\Upsilon},\\
            & z(x,2)=0,\mbox{ in }\Omega,\\
            & z(x,t)=0,\mbox{ on }\partial\Omega\times [2,4],\\
\end{aligned}
\right.
\end{equation}
where $\widetilde{\Upsilon}=\Omega\times[2,4]$, $x=(x_{1},x_{2})^{\top}$ and $\triangle=\frac{\partial^{2}}{\partial x_{1}^{2}}+\frac{\partial^{2}}{\partial x_{2}^{2}}$, the two dimensions Laplace operator.
Here, $A=-\triangle$ generates a $C_{0}-$semigroup $\{T(t)\}$ on $L^{2}(\Omega)$ and $-A=\triangle$ is a uniformly elliptic operator.
As we know, the eigenvalue $\lambda_{kl}$ and the corresponding eigenfunction $\alpha_{kl}(x)$ of $A$ are $\lambda_{kl}=\left(k^{2}+l^{2}\right)\pi^{2}$ and $\alpha_{kl}(x)=2\sin(k\pi x_{1})\sin(l\pi x_{2})$, respectively, for $k,l=1,\,2,\ldots$ and $x\in\Omega$.

According to \dref{13}, for any $z\in L^{2}(\Omega)$, it follows that
\begin{equation*}
\begin{aligned}
H^{*}\nabla^{*}z=\frac{1}{t}N(t)\sum^{\infty}_{k,l=1} &\,E_{0.5,\,0.5}\left(\lambda_{kl}N(t)\right)
\left\langle\nabla^{*}z,\,\alpha_{kl}(x)\right\rangle M_{kl},\\
\end{aligned}
\end{equation*}
where $N(t)=\left(\log\frac{4}{t}\right)^{-0.5}$ and $M_{kl}=\int _{P} \alpha_{kl}(x)dx$.
When we consider the gradient approximate controllability on the whole region at time 4, that is, $P=\Omega$, then
\begin{equation*}
M_{kl}=\int^{1}_{-1}\int^{1}_{-1} 2\sin(k\pi x_{1})\sin(l\pi x_{2})dx_{1}dx_{2}\equiv0.
\end{equation*}
Hence, we have
\begin{equation}
H^{*}\nabla^{*}z\equiv0,\,\,\,\,\,\,\,\,\forall z\in L^{2}(\Omega),
\end{equation}
which implies that $Ker\left(H^{*}\nabla^{*}\right)\neq \{0\}$, i.e., $\overline{Im(\nabla H)}\neq L^{2}(\Omega)$. From Theorem \ref{3.2}, we conclude that system \dref{eg} is not gradient approximately controllable on the whole interested domain at time 4.

Next, we show that system \dref{eg} is regionally gradient approximately controllable on a subinterval $\omega\subsetneqq \Omega$.

Let $\omega=[0,1]\times[0,1]$ and $z=\sin(p\pi x_{1})\cos(q\pi x_{2})$ with even numbers $p,q$. Then
\begin{equation*}
H^{*}\nabla^{*}p^{*}_{\omega}z= \frac{1}{t}N(t)\sum_{k,l}E_{0.5,\,0.5}\left(-\left(k^{2}+l^{2}\right)\pi^{2}N(t)\right)J_{klpq}\neq0,
\end{equation*}
where
\begin{equation*}
J_{klpq}=\frac{8p}{kl\pi}\left(\frac{1}{(k+p)\pi}-\frac{1}{(k-p)\pi}\right)\left(\frac{1}{(l+q)\pi}-\frac{1}{(l-q)\pi}\right),
\end{equation*}
for odd numbers $k,l$.

Hence, $z=\sin(p\pi x_{1})\cos(q\pi x_{2})$ is reachable and thus regionally gradient approximately controllable on $\omega$ at time 4.
Since the eigenvalues of $A$ are all of multiplicity 1, that is, $r_{k}=1$ for all $k$, then $r=1$ and $rank D_{k}^{1}=rank D_{k}^{2}=1$. Hence, from Theorem \ref{3.3},  the zone actuator is gradient $\omega-$stratigic, which coincides with Remark \ref{3}.
Moreover, the regional gradient approximate controllability of system \dref{eg} on $\omega$ at time 4 shows that \dref{30} defines a norm on G according to Lemma \ref{3.4}. It follows from $\left\langle Bu,\varphi(x,t)\right\rangle=\left\langle u,B^{*}\varphi(x,t)\right\rangle$ and $Bu=\chi_{\omega}u(t)$ that
\begin{equation*}
B^{*}\varphi(x,t)=\int_{L^{2}(\omega)}\varphi(x,t)dx.
\end{equation*}
Therefore, according to Theorem \ref{3.5}, the unique minimum energy control can be presented by
\begin{equation*}
u^{*}(t)=\frac{1}{t}\int_{L^{2}(\omega)}\varphi(x,t)dx.
\end{equation*}

\section{Conclusion}
  \label{sec:con}

In this paper, we established some effective necessary and sufficient conditions on the regional gradient controllability of Hadamard-Caputo time fractional diffusion systems.
The characteristics of admissible actuators and the optimal control described by the minimum energy functional were also derived.
Finally, the illustrative example showed the application of our results in practical models.

\section{Further works}
  \label{sec:fur}
\cite{A53} considered a rather stronger notion of null controllability, which requires the state of system stays at rest after the final moment. In this sense, a fractional system is not null controllable.
It's natural to ask whether it is controllable in other sense, such as
regionally gradient approximately controllable.
Similar to the definition of null controllability in \cite{A53}, one should take into account the memory effect of the fractional derivative for the regional gradient approximate controllability.
This is beyond the scope of this paper and will be analyzed in detail in another paper.

\section*{Acknowledgments}
We would like to thank the referees for all the comments and
suggestions that made possible a better version of this paper.


\medskip
Received  July 2018; revised February 2019.
\medskip

\end{document}